\documentclass[10pt; a4paper]{amsart} 
\usepackage{amscd, amsfonts, amsmath, amssymb, amsthm, fixmath, mathrsfs}
\usepackage{hyperref}
\usepackage[all]{xy} 
\makeindex

\theoremstyle{definition} 
\newtheorem{Unity}{Unity}[section] 
\newtheorem*{Definition*}{Definition} 
\newtheorem{Definition}[Unity]{Definition} 

\theoremstyle{plain} 
\newtheorem*{Theorem*}{Theorem}
\newtheorem{Theorem}[Unity]{Theorem}
\newtheorem{Proposition}[Unity]{Proposition}
\newtheorem{Corollary}[Unity]{Corollary}
\newtheorem{Lemma}[Unity]{Lemma}

\theoremstyle{remark} 
\newtheorem*{Remark*}{Remark}
\newtheorem{Remark}[Unity]{Remark}

\numberwithin{Unity}{section}

\newcommand{\N}{\mathbb{N}}
\newcommand{\Z}{\mathbb{Z}}

\newcommand{\E}{\mathscr{E}}
\newcommand{\F}{\mathscr{F}}
\newcommand{\Hs}{\mathscr{H}}
\newcommand{\Ls}{\mathscr{L}}

\newcommand{\Ox}{\mathscr{O}}
\newcommand{\rk}{\mathrm{rk}}
\newcommand{\Pic}{\mathrm{Pic}}

\newcommand{\Omg}{\mathrm{\Omega}}

\newcommand{\depth}{\mathrm{depth}}

\begin{document}

\title{Strong Stability of Cotangent Bundles of Cyclic Covers}
\author{Lingguang Li}
\author{Junchao Shentu}
\address{Department of Mathematics, Tongji University, Shanghai, P. R. China}
\email{LG.Lee@amss.ac.cn}
\address{Academy of Mathematics and Systems Science, Chinese Academy of Science, Beijing, P. R. China}
\email{stjc@amss.ac.cn}
\thanks{Partially supported by the National Natural Science Foundation (No. 11271275).}
\begin{abstract}
Let $X$ be a smooth projective variety over an algebraically closed field $k$ of characteristic $p>0$ of $\dim X\geq 4$ and Picard number $\rho(X)=1$. Suppose that $X$ satisfies $H^i(X,F^{m*}_X(\Omg^j_X)\otimes\Ls^{-1})=0$ for any ample line bundle $\Ls$ on $X$, and any nonnegative integers $m,i,j$ with $0\leq i+j<\dim X$, where $F_X:X\rightarrow X$ is the absolute Frobenius morphism. Let $Y$ be a smooth variety obtained from $X$ by taking hyperplane sections of dim $\geq 3$  and
cyclic covers along smooth divisors. If the canonical bundle $\omega_Y$ is ample (resp. nef), then we prove that $\Omg_Y$ is strongly stable $($resp. strongly semistable$)$ with respect to any polarization.
\end{abstract}
\maketitle

\section{Introduction}

An important outstanding problem in differential geometry is asking whether the tangent bundles admit Hermitian-Einstein metrics. By Kobayashi-Hitchin correspondence, this problem is related to the stability of tangent bundles. In algebraic geometry over positive characteristic, there exists another useful notion of strong stability of sheaves. X. Sun \cite{Sun08}\cite{Sun10}, G. Li and F. Yu \cite{Li13} have showed that the strong stability of cotangent bundles has relation with the stability of Frobenius direct image of sheaves. So we would like to know which classes of varieties have strongly semistable cotangent bundles in positive characteristic. However, as far as I know that there are only a few classes of varieties with strongly semistable cotangent bundles that have been found. K. Joshi \cite{Joshi00} showed that the cotangent bundles of the general type hypersurfaces of $\mathbb{P}_k^n$ ($n\geq4$) are strongly stable. A. Noma \cite{Noma97} \cite{Noma01} proved that any smooth weighted complete intersection $X$ of some weak projective space with $\Pic(X)\cong\Z$ has strongly stable cotangent bundle. Later I. Biswas \cite{Biswas10} given some conditions under which the cotangent bundles of complete intersections on some Fano varieties are strongly stable.

The motivation of this paper is to find new classes of varieties with strongly (semi)stable cotangent bundles in positive characteristic. T. Peternell and J. Wi\'{s}niewski \cite{PeternellWisniewski95} have studied the stability of cotangent bundles of hypersurfaces and cyclic covers over complex field. We study the strong stability of cotangent bundles of hypersurfaces and cyclic covers in positive characteristic. The main result is:

\begin{Theorem*}
Let $k$ be an algebraically closed field of characteristic $p>0$, $X$ a $n(\geq 4)$-dimensional smooth projective variety of Picard number $\rho(X)=1$ over $k$. Suppose that $X$ satisfies $H^i(X,F^{m*}_X(\Omg^j_X)\otimes\Ls^{-1})=0$ for any ample line bundle $\Ls$ on $X$, any nonnegative integers $m,i,j$ with $0\leq i+j<n$. Let $Y$ be a smooth variety obtained from $X$ by taking hyperplane sections of dim $\geq 3$  and cyclic covers along smooth divisors. If the canonical bundle $\omega_Y$ is ample (resp. nef), then $\Omg_Y$ is strongly stable $($resp. strongly semistable$)$.
\end{Theorem*}

As an application, let $X$ be a $n(\geq4)$-dimensional smooth weighted complete intersection of some weak projective space. Then the cotangent bundles of smooth ample general type (resp. non-Fano) divisors and cyclic coves along smooth ample divisors of $X$ are strongly stable (resp. strongly semistable). (See Corollary \ref{Cor:WCI}).

The paper is organized as follows.
In section 2 we recall some definitions and Grothendieck-Lefschetz theorem on Picard groups in arbitrary characteristic (Lemma \ref{LefschetzThm}), which is crucial for our proofs.
In section 3 we introduce the notion of Frobenius vanish of varieties in positivie characteristic (Definition \ref{CSV}), and prove that under wide condition this property is preserved under taking hypersurfaces and cyclic covers (Proposition \ref{CSV_for_divisor_2}, Theorem \ref{CSV_for_cover}).
In section 4 we show that Frobenius vanish of varieties induces the strong (semi-)stability of cotangent bundles (Theorem \ref{strongstable}) and obtain the main result of this paper (Theorem \ref{Thm:StrongSTCover}).

\textbf{Acknowledgements:} We want to thank Professor Xiaotao Sun for his encouragement. We would also like to express our hearty thanks to Professor Baohua Fu, who helped us to translate the abstract into French and gave helpful comments to our manuscript.

\section{Preliminary}

Let $k$ be an algebraically closed field of characteristic $p>0$, $X$ a smooth projective variety of dimension $n$ over $k$ with fixed ample divisors $\Hs:=\{H_1,\cdots, H_{n-1}\}$. The \emph{absolute Frobenius morphism} $F_X:X\rightarrow X$ is induced by $\Ox_X\rightarrow\Ox_X$, $f\mapsto f^p$, with identity on the underlining topological space. Let $\E$ be a torsion free sheaf on $X$, the \emph{$\Hs$-slope} of $\E$ is defined as $$\mu_{\Hs}(\E):=\frac{c_1(\E)\cdot H_1\cdots H_{n-1}}{\rk(\E)}.$$
Then $\E$ is called \emph{$\Hs$-stable} (resp. \emph{$\Hs$-semistable}) if $\mu_{\Hs}(\F)<(\text{resp.} \leq)\mu_{\Hs}(\E)$ for any nonzero subsheaf $\F\subsetneq\E$ with $\rk(\F)<\rk(\E)$. If $F^{m*}_X(\E)$ is $\Hs$-stable (resp. $\Hs$-semistable) for any nonnegative integer $m\in\N$, then $\E$ is called \emph{strongly $\Hs$-stable} (resp. \emph{strongly $\Hs$-semistable}).

Let $\Ls$ be a line bundle on $X$, $D\in|\Ls^d|$ a smooth divisor on $X$ for some positive integer number $d>0$, $(p,d)=1$ (Unless stated otherwise, we always require $(p,d)=1$ in construction of cyclic cover). Let $\pi:Y\rightarrow X$ be the cyclic covering over $X$ branched along $D$ (Without confusion, we will omit mentioning $\Ls$ in construction), then there is a smooth divisor $D'\in|\pi^*(\Ls)|$ maps isomorphically to $D$. Let $\Omg^1_X(\log D)$ (resp. $\Omg^1_Y(\log D')$) the sheaf of one-forms on $X$ (resp. Y) with logarithmic pole $D$ (resp. $D'$). Then we have isomorphism $$\Omg^j_Y(\log D')\cong\pi^*(\Omg^j_X(\log D)),$$ which gives adjunction formula $\omega_Y\cong\pi^*(\omega_X\otimes\Ls^{d-1})$, and the following exact sequences ($1\leq j\leq n$)
$$0\longrightarrow\Omg^j_X\longrightarrow\Omg^j_X(\log D)\longrightarrow i_*\Omg^{j-1}_D\longrightarrow0,$$
$$0\longrightarrow\Omg^j_Y\longrightarrow\Omg^j_Y(\log D')\longrightarrow i'_*\Omg^{j-1}_{D'}\longrightarrow0.$$
where $i: D\rightarrow X$ and $i': D'\rightarrow Y$ are the canonical embeddings.

\begin{Lemma}\cite[Expos\'e XII Corollary 3.6]{Grothendieck68}\label{LefschetzThm}
Let $X$ be a projective scheme over a field $k$, $D\subset X$ an ample Cartier divisor. Assume $\depth~D_x\geq 3$ for any closed points $x\in D$. Moreover, if $X\setminus D$ is regular and $H^i(D,\Ox_D(-lD))=0$ for any integer $l>0$, $i=1,2$. Then the restriction map $\Pic(X)\rightarrow\Pic(D)$ is an isomorphism.
\end{Lemma}

\section{Frobenius $\Hs$-vanish property}

Now we introduce the notion of Frobenius $\Hs$-vanish for projective varieties in positive characteristic.

\begin{Definition}\label{CSV}
Let $k$ be an algebraically closed field of characteristic $p>0$, and $X$ a smooth projective variety of dimension $n$ over $k$. Fix ample divisors $\Hs=\{H_1,\cdots,H_{n-1}\}$ on $X$. A line bundle $\Ls$ on $X$ is called \emph{$\Hs$-positive} $($resp. \emph{$\Hs$-nonnegative}$)$ if $$c_1(\Ls)\cdot H_1\cdots H_{n-1}>(\textit{resp.}\geq)0.$$ We call $X$ has \emph{Frobenius $\Hs$-vanish} in level $m\in\N$ up to rank $N\in\N_+$ if for any $\Hs$-positive line bundle $\Ls$ on $X$, any nonnegative integers $i,j$ with $0\leq i+j<N$, we have $$H^i(X,F_X^{m*}(\Omg^j_X)\otimes\Ls^{-1})=0.$$
\end{Definition}

\begin{Remark}\label{Rmk:WCI}
Any ample line bundle is $\Hs$-positive, $\Hs$-positive (resp. $\Hs$-nonnegative) is equivalent to ample (resp. nef) if $X$ is of Picard number $1$. Hence, by \cite[Proposition 2.1]{Noma01}, any $n(\geq 3)$-dimensional smooth weighted complete intersections of weak projective spaces have Frobenius $\Hs$-vanish up to rank $n$ in any level.
\end{Remark}

\begin{Lemma}\label{CSV_for_divisor}
Let $k$ be an algebraically closed field of characteristic $p>0$, $m\in\N$, $N\in\N_+$, $X$ a smooth projective variety over $k$ having Frobenius $\Hs$-vanish in level $m$ up to rank $N$. Let $D$ be a smooth $\Hs$-nongenative effective divisor of $X$. Then for any $\Hs$-positive line bundle $\Ls$ on $X$, any nonnegative integers $i,j$ with $0\leq i+j<N-1$, we have $H^i(D,F^{m*}(\Omg^j_D)\otimes(\Ls|_D)^{-1})=0$.
\end{Lemma}
\begin{proof}
Consider the exact sequence of sheaves
$$0\longrightarrow\Omg^j_X\otimes\mathcal{O}_X(-D)\longrightarrow\Omg^j_X\longrightarrow\Omg^j_X|_D\longrightarrow0,$$
which is obtained by tensoring the exact sequence of sheaves on $X$
$$0\longrightarrow\mathcal{O}_X(-D)\longrightarrow\Ox_X\longrightarrow\Ox_D\longrightarrow0.$$
with $\Omg^j_X$ ($1\leq j\leq n$). Applying $F^{m*}_X$ to above sequence and tensoring with $\Ls^{-1}$, where $\Ls$ is a $\Hs$-positive line bundle on $X$. Then we have exact sequence
$$0\longrightarrow F^{m*}_X(\Omg^j_X)\otimes\mathcal{O}_X(-D)^{p^{nm}}\otimes\Ls^{-1}\longrightarrow F^{m*}_X(\Omg^j_X)\otimes\Ls^{-1}\longrightarrow F^{m*}_X(\Omg^j_X|_D)\otimes\Ls^{-1}\longrightarrow0,$$
and this deduces an exact sequence of cohomology groups
$$\cdots\longrightarrow H^i(X,F^{m*}_X(\Omg^j_X)\otimes\Ls^{-1})\longrightarrow H^i(X,F^{m*}_X(\Omg^j_X|_D)\otimes\Ls^{-1})$$
$$\longrightarrow H^{i+1}(X,F^{m*}_X(\Omg^j_X)\otimes(\mathcal{O}_X(D)^{p^{nm}}\otimes\Ls)^{-1})\longrightarrow\cdots.$$
Then by the Frobenius $\Hs$-vanish property of $X$, for any nonnegative integers $i,j$ with $0\leq i+j<N-1$, we have $H^i(X,F^{m*}_X(\Omg^j_X|_D)\otimes\Ls^{-1})=0$.

Consider the exterior power of exact sequence of cotangent-conormal sheaves
$$0\longrightarrow\mathcal{O}_X(-D)|_D\longrightarrow\Omg^1_X|_D\longrightarrow\Omg^1_D\longrightarrow 0,$$
we obtain exact sequence
$$0\longrightarrow\Omg^{j-1}_D\otimes\mathcal{O}_X(-D)|_D\longrightarrow\Omg^j_X|_D\longrightarrow\Omg^j_D\longrightarrow 0$$
for any integer $1\leq j\leq n$.
Applying $F^{m*}_D$ to above sequence and tensoring with line bundle $(\Ls|_D)^{-1}$, we have
$$0\rightarrow F^{m*}_D(\Omg^{j-1}_D)\otimes(\mathcal{O}_X(-D)^{p^m}\otimes\Ls^{-1})|_D\rightarrow F^{m*}_D(\Omg^j_X|_D)\otimes(\Ls|_D)^{-1}\rightarrow F^{m*}_D(\Omg^j_D)\otimes(\Ls|_D)^{-1}\rightarrow 0.$$
Then we have an exact sequence of cohomology groups
$$\cdots\longrightarrow H^i(D,F^{m*}_D(\Omg^j_X|_D)\otimes(\Ls|_D)^{-1})\longrightarrow H^i(D,F^{m*}_D(\Omg^j_D)\otimes(\Ls|_D)^{-1})$$
$$\longrightarrow H^{i+1}(D,F^{m*}_D(\Omg^{j-1}_D)\otimes(\Ox_X(-D)^{p^m}\otimes\Ls^{-1})|_D)\longrightarrow\cdots.$$
Since $\Ox_X(D)^{p^m}\otimes\Ls$ is $\Hs$-positive on $X$, then $H^i(D,(\Ox_X(-D)^{p^m}\otimes\Ls^{-1})|_D)=0$ for any integer $0\leq i< N-1$. Hence, using induction on above sequence, we have $H^i(D,F^{m*}(\Omg^j_D)\otimes(\Ls|_D)^{-1})=0$ for any $\Hs$-positive line bundle $\Ls$ on $X$ and any nonnegative integers $i,j$ with $0\leq i+j<N-1$.
\end{proof}

\begin{Corollary}\label{GroLefThm}
Let $k$ be an algebraically closed field, $X$ a $n(\geq4)$-dimensional smooth projective variety over $k$. Suppose that $X$ has Frobenius $\Hs$-vanish in level $0$ up to rank $N(\geq 3)$. Let $D$ be a smooth ample effective divisor of $X$. Then the restriction map $\Pic(X)\rightarrow\Pic(D)$ is an isomorphism.
\end{Corollary}

\begin{proof}
By lemma \ref{CSV_for_divisor}, we have $H^i(D,\Ox_D(-nD))=0$ for any integer $n>0$, $i=1,2$. Hence by lemma \ref{LefschetzThm}, the restriction map $\Pic(X)\rightarrow\Pic(D)$ is an isomorphism.
\end{proof}

\begin{Proposition}\label{CSV_for_divisor_2}
Let $k$ be an algebraically closed field of characteristic $p>0$, $m\in\N$, $X$ a $n(\geq 4)$-smooth projective variety over $k$ of Picard number $\rho(X)=1$. If $X$ has Frobenius $\Hs$-vanish in level $m$ up to rank $N(\geq3)$. Then any smooth ample effective divisors of $X$ have Frobenius $\Hs$-vanish in level $m$ up to rank $N-1$.
\end{Proposition}
\begin{proof}
Since $\rho(X)=1$. All $\Hs$-positive line bundles are ample, and all ample line bundles are numerical equivalent up to a positive scalar.
By, corollary \ref{GroLefThm}, $D$ is also of Picard number $\rho(D)=1$. Therefore, any line bundle on $D$ is of the form $\Ls|_D$, where $\Ls$ is a line bundle on $X$. But $\Ls$ is ample on $X$ if and only if $\Ls|_D$ is ample on $D$. Hence this proposition follows from lemma \ref{CSV_for_divisor}.
\end{proof}

\begin{Theorem}\label{CSV_for_cover}
Let $k$ be an algebraically closed field of characteristic $p>0$, $m\in\N$, $N\in\N_+$, and $X$ a $n(\geq4)$-dimensional smooth projective variety over $k$, $D$ a smooth ample divisor on $X$, and $\pi:Y\rightarrow X$ a cyclic cover of $X$ branched along $D$. Suppose that $X$ has Frobenius $\Hs$-vanish in level $m$ up to rank $N(\geq3)$. Then $Y$ also has Frobenius $\Hs$-vanish in level $m$ up to rank $N$.
\end{Theorem}
\begin{proof}
By construction of cyclic cover, there exists smooth divisor $i':D'\rightarrow Y$ maps isomorphically to $D$. This implies $H^i(D,\Ox_D(-lD))\cong H^i(D',\Ox_{D'}(-lD'))$. Moreover, by corollary \ref{GroLefThm}, the commutative diagram
$$\xymatrix{
D' \ar[r]^{i'}\ar[d]^{\cong} & Y \ar[d]^{\pi}\\
D \ar[r]^i & X,}$$
deduces the following commutative diagram
$$\xymatrix{
\Pic(D') & \Pic(Y) \ar[l]_{\cong}\\
\Pic(D) \ar[u]_{\cong} & \Pic(X). \ar[l]_{\cong}\ar[u]_{\pi^*}}$$
This implies the homomorphism $\pi^*:\Pic(X)\rightarrow\Pic(Y)$ is an isomorphism. Let $\Ls\in\Pic(X)$, then $\Ls$ is $\Hs$-positive (resp. $\Hs$-nonnegative) on $X$ if and only if $\pi^*\Ls$ is $\pi^*\Hs$-positive (resp. nonnegative) on $Y$, where $\pi^*(\Hs):=\{\pi^*(H_1),\cdots,\pi^*(H_{n-1})\}$.

Applying $F^{m*}_{X}$ to the following sequence and tensoring with $\Ls^{-1}$,
$$0\longrightarrow\Omg^j_X\longrightarrow\Omg^j_X(\log D)\longrightarrow i_*\Omg^{j-1}_D\longrightarrow0,$$
where $\Ls$ is a $\Hs$-positive line bundle on $X$. We have exact sequence
$$0\longrightarrow F^{m*}_X(\Omg^j_X)\otimes\Ls^{-1}\longrightarrow F^{m*}_X(\Omg^j_X(\log D))\otimes\Ls^{-1}\longrightarrow F^{m*}_X(i_*\Omg^{j-1}_D)\otimes\Ls^{-1}\longrightarrow0.$$
This deduces exact sequence of cohomology groups
$$\cdots\longrightarrow H^i(X,F^{m*}_X(\Omg^j_X)\otimes{\Ls}^{-1})\longrightarrow H^i(X,F^{m*}_X(\Omg^j_X(\log D))\otimes\Ls^{-1})$$
$$\longrightarrow H^i(X,F^{m*}_X(i_*\Omg^{j-1}_D)\otimes\Ls^{-1})\longrightarrow\cdots.$$
Notice that $F_X^{m*}i_*\cong i_* F_D^{m*}$, there exists natural isomorphism
$$F^{m*}_X(i_*\Omg^{j-1}_D)\otimes\Ls^{-1}\cong i_*(F^{m*}_{D}(\Omg^{j-1}_D)\otimes(\Ls|_D)^{-1}).$$
Then, by Lemma \ref{CSV_for_divisor}, for any nonnegative integers $i,j$ with $0\leq i+j<N$, we have $H^i(X,F^{m*}_X(\Omg^{j-1}_D)\otimes\Ls^{-1})=0$. Hence $H^i(X,F^{m*}_X(\Omg^j_X(\log D))\otimes\Ls^{-1})=0$.

From the isomorphism $\Omg^j_Y(\log D')\cong\pi^*(\Omg^j_X(\log D))$, we have isomorphisms $$F^{m*}_Y(\Omg^1_Y(\log D'))\cong F^{m*}_Y(\pi^*(\Omg^j_X(\log D)))\cong\pi^*(F^{m*}_X(\Omg^j_X(\log D))).$$
As $\pi_*(\Ox_Y)\cong\bigoplus_{0\leq i<d}\Ox(D)^{-i}$, so by projective formula, we get isomorphism $$\pi_*(F^{m*}_Y(\Omg^j_Y(\log D')))\cong F^{m*}_X(\Omg^j_X(\log D))\otimes\bigoplus_{0\leq i<d}\Ox(D)^{-i}.$$

Since $\pi:Y\rightarrow X$ is an affine morphism, by projective formula, we have
\begin{eqnarray*}
H^i(Y,F^{m*}_Y(\Omg^j_Y(\log D'))\otimes\pi^*(\Ls^{-1}))\cong H^i(Y,F^{m*}_X(\Omg^j_X(\log D))\otimes(\bigoplus_{0\leq i<d}\Ox(D)^{-i})\otimes\Ls^{-1})=0.
\end{eqnarray*}

Consider the following exact sequence on $Y$ ($1\leq j\leq n$)
$$0\longrightarrow\Omg^j_Y\longrightarrow\Omg^j_Y(\log D')\longrightarrow i'_* \Omg^{j-1}_{D'}\longrightarrow0.$$
Applying $F^{m*}_{Y}$ to above sequence and tensoring with $\pi^*(\Ls^{-1})$, we get exact sequece
$$0\longrightarrow F^{m*}_Y(\Omg^j_Y)\otimes\pi^*(\Ls^{-1})\longrightarrow F^{m*}_Y(\Omg^j_Y(\log D'))\otimes\pi^*(\Ls^{-1})\longrightarrow F^{m*}_Y(i'_*\Omg^{j-1}_{D'})\otimes\pi^*(\Ls^{-1})\longrightarrow0,$$
Notice that $F_Y^{m*}i'_*\cong i'_*F_{D'}^{m*}$, there exists natural isomorphism
$$F^{m*}_Y(i'_*\Omg^{j-1}_{D'})\otimes\pi^*(\Ls^{-1})\cong i_*(F^{m*}_{D'}(\Omg^{j-1}_{D'})\otimes\pi^*(\Ls^{-1})|_{D'}).$$
Taking cohomology we have
$$\cdots\longrightarrow H^{i-1}(D',F^{m*}_{D'}(\Omg^{j-1}_{D'})\otimes\pi^*(\Ls^{-1})|_{D'})\longrightarrow H^i(Y,F^{m*}_Y(\Omg^j_Y)\otimes\pi^*(\Ls^{-1}))$$
$$\longrightarrow H^i(Y,F^{m*}_Y(\Omg^j_Y(\log D'))\otimes\pi^*(\Ls^{-1}))\longrightarrow\cdots.$$
Since $D'$ maps isomorphically to $D$ with commutative diagram
$$\xymatrix{
D' \ar[r]^{i'}\ar[d]^{\cong} & Y \ar[d]^{\pi}\\
D \ar[r]^i & X.}$$
Therefore, there is isomorphism $$H^{i-1}(D',F^{m*}_{D'}(\Omg^{j-1}_{D'})\otimes\pi^*(\Ls^{-1})|_{D'})\cong H^{i-1}(D,F^{m*}_D(\Omg^{j-1}_D)\otimes(\Ls|_D)^{-1}).$$
Then, by lemma \ref{CSV_for_divisor}, for any nonnegative integers $i,j$ with $0\leq i+j<N$, we have $H^i(Y,F^{m*}_Y(\Omg^j_Y)\otimes\pi^*(\Ls^{-1}))=0$. This completes the proof of this theorem.
\end{proof}

\section{Strong Stability of Cotangent Bundles}

Frobenius $\Hs$-vanish property has closed relation with strong stability of cotangent bundles in positive characteristic, at least for smooth projective varieties of Kodaira dimension $\geq 0$.

\begin{Proposition}\label{strongstable}
Let $k$ be an algebraically closed field of characteristic $p>0$, $m\in\N$, and $X$ a smooth projective variety over $k$. Suppose that $X$ has Frobenius $\Hs$-vanish in level $m$ up to rank $\dim X$ and $\omega_X$ is $\Hs$-positive $($resp. $\Hs$-nonnegative$)$. Then $F_X^{m*}(\Omg_X)$ is $\Hs$-stable $($resp. $\Hs$-semistable$)$.
\end{Proposition}
\begin{proof}
This is a classical argument. Assume $F_X^{m*}(\Omg_X)$ is not $\Hs$-stable (resp. $\Hs$-semistable), then there is a reflexive subsheaf $\E\subseteq F_X^{m*}(\Omg_X)$ of rank $j(<\dim X)$ such that $\mu_{\Hs}(\E)\geq(\text{resp. }>)\mu_{\Hs}(F_X^{m*}(\Omg_X))$. This induces a nontrivial homomorphism $\det(\E)\rightarrow F_X^{m*}(\Omg^j_X)$. Since $\omega_X$ is $\Hs$-positive $($resp. $\Hs$-nonnegative$)$, we have $\det(\E)$ is $\Hs$-positive. This contradicts the Frobenius $\Hs$-vanish assumption on $X$. Hence, $F_X^{m*}(\Omg_X)$ is $\Hs$-stable $($resp. $\Hs$-semistable$)$.
\end{proof}

\begin{Remark}
If $\kappa(X)>0$ (resp. $\geq 0$), then $\omega_X$ is $\Hs$-positive (resp. $\Hs$-nonnegative).
\end{Remark}

\begin{Corollary}
Let $k$ be an algebraically closed field of characteristic $p>0$, $m\in\N$, and $X$ a $n(\geq4)$-dimensional smooth projective variety over $k$, $\pi:Y\rightarrow X$ a cyclic cover of $X$ branched along a smooth ample divisor $D$. Suppose that $X$ has Frobenius $\Hs$-vanish in level $m$ up to rank $n$, and $\omega_Y$ is $\pi^*(\Hs)$-positive $($resp. $\pi^*(\Hs)$-nonnegative$)$. Then $F_X^{m*}(\Omg_Y)$ is $\pi^*(\Hs)$-stable $($resp. $\pi^*(\Hs)$-semistable$)$.
\end{Corollary}
\begin{proof}
It is obviously by theorem \ref{CSV_for_cover} and proposition \ref{strongstable}.
\end{proof}

Combine all results above we obtain our main result of the paper.

\begin{Theorem}\label{Thm:StrongSTCover}
Let $k$ be an algebraically closed field of characteristic $p>0$, $X$ a $n(\geq 4)$-dimensional smooth projective variety of Picard number $\rho(X)=1$ over $k$. Suppose that $X$ has Frobenius $\Hs$-vanish of rank $n$ in any level. Let $Y$ be a smooth variety obtained from $X$ by taking hyperplane sections of dim $\geq 3$  and cyclic covers along smooth divisors. If the canonical bundle $\omega_Y$ is ample (resp. nef), then $\Omg_Y$ is strongly stable $($resp. strongly semistable$)$ with respect to any polarization.
\end{Theorem}
\begin{proof}
By corollary \ref{GroLefThm} and proof of theorem \ref{CSV_for_cover}, we have $\rho(Y)=1$. Hence, any polarization are numerical equivalent up to a positive scalar. Combine proposition \ref{CSV_for_divisor_2}, theorem \ref{CSV_for_cover} and proposition \ref{strongstable} we get the strong stability of $\Omg_Y$ with respect to any polarization.
\end{proof}

\begin{Corollary}\label{Cor:WCI}
Let $k$ be an algebraically closed field of characteristic $p>0$, and $X$ a $n(\geq 4)$-dimensional smooth weighted complete intersection of a weak projective space. Let $Y$ be a smooth variety obtained from $X$ by taking hyperplane sections of dim $\geq 3$  and cyclic covers along smooth divisors. If the canonical bundle $\omega_Y$ is ample (resp. nef), then $\Omg_Y$ is strongly stable $($resp. strongly semistable$)$ with respect to any polarization.
\end{Corollary}
\begin{proof}
It is easily follows from theorem \ref{Thm:StrongSTCover} and Remark \ref{Rmk:WCI}.
\end{proof}

\end{document}